\documentclass{amsart}
\usepackage{amsmath}
\usepackage{bbm}
\usepackage{latexsym}
\usepackage{xcolor}
\usepackage{hyperref}
\usepackage{verbatim} 
\usepackage{mathtools} 
\usepackage[mathcal]{euscript}

\newcommand{\R}{\mathbb{R}}
\newcommand{\N}{\mathbb{N}}

\newcommand{\mc}[1]{\mathcal{#1}}

\newcommand{\ovr}[1]{\overrightarrow{#1}}

\newtheorem{thm}{Theorem}
\newtheorem{theorem}[thm]{Theorem}
\newtheorem{lemma}[thm]{Lemma}

\newtheorem{prop}[thm]{Proposition}
\newtheorem{corollary}[thm]{Corollary}
\newtheorem{proposition}[thm]{Proposition}

\newtheorem*{thmi}{Theorem}

\theoremstyle{definition}
\newtheorem{defi}[thm]{Definition}

\theoremstyle{definition}

\theoremstyle{remark} 
\newtheorem{remark}[thm]{Remark}

\newcommand{\be}{\begin{equation}}
\newcommand{\ee}{\end{equation}}

\mathtoolsset{showonlyrefs,showmanualtags} 
\numberwithin{equation}{section}

\title[Fuller singularities of generic control-affine systems]{
Time-optimal trajectories of generic control-affine systems have at worst iterated Fuller singularities 
}
\date{\today}

\author{Francesco Boarotto}
\address{Laboratorie Jacques-Louis Lions, Sorbonne Universit\'e, Universit\'e Paris-Diderot SPC, CNRS, Inria, France}
\email{\href{mailto:boarotto@ljll.math.upmc.fr}{\nolinkurl{boarotto@ljll.math.upmc.fr}}}

\author{Mario Sigalotti}
\address{Inria \& Laboratorie Jacques-Louis Lions, Sorbonne Universit\'e, Universit\'e Paris-Diderot SPC, CNRS, Inria, France}
\email{\href{mailto:Mario.Sigalotti@inria.fr}{\nolinkurl{mario.sigalotti@inria.fr}}}

\begin{document}
	\maketitle
	
\begin{abstract}
We consider in this paper the regularity problem for time-optimal trajectories of a single-input control-affine system on a $n$-dimensional manifold. 
We prove that, under generic conditions on the drift and the controlled vector field, any control $u$ associated with an optimal trajectory
is smooth out of a countable set of times. More precisely, there exists an integer $K$, only depending on the dimension $n$, such that the non-smoothness set of $u$ is made of isolated points, accumulations of isolated points,  
and so on up to $K$-th order iterated accumulations.   
\end{abstract}	
	
	\section{Introduction}\label{sec:intro}
	
	\subsection{Single-input systems and chattering phenomena} Let $M$ be a smooth\footnote{i.e., $C^\infty$ throughout the whole paper.}, connected, 
	$n$-dimensional manifold and denote by $\mathrm{Vec}(M)$ the space of smooth vector fields on $M$. Consider the (single-input) control-affine system
	\be\label{eq:si} \dot{q}=f_0(q)+uf_1(q),\quad q\in M,\quad u\in[-1,1],\quad f_0,f_1\in\mathrm{Vec}(M).\ee
	An \emph{admissible trajectory} of \eqref{eq:si} is an absolutely continuous curve 
	$q:[0,T]\to M$, $T>0$,
	such that 
	there exists $u\in L^\infty([0,T],[-1,1])$ so that 
	$\dot q(t)=f_0(q(t))+u(t)f_1(q(t))$ for almost every $t\in[0,T]$. 
	
	For any fixed initial datum $q_0\in M$, the time-optimal control problem associated with \eqref{eq:si} consists into looking for admissible trajectories $q:[0,T]\to M$, $T>0$, that minimize the time needed to steer $q_0$ to $q(T)$ among all admissible trajectories.
	
	A necessary (but not sufficient) condition for an admissible trajectory to be time-optimal is provided by the Pontryagin maximum principle (PMP, in short) \cite{PMP}. Introducing the control-dependent Hamiltonian \be\label{eq:maxH}\mathcal{H}:T^*M\times [-1,1]\to\R,\quad \mathcal{H}(\lambda,v)=\langle \lambda,(f_0+v f_1)(q)\rangle,\quad q=\pi(\lambda),\ee the PMP states that if  a trajectory $q(\cdot)$ associated with the 
	control $u(\cdot)$ is time-optimal, then it is \emph{extremal}, i.e., there exists an absolutely continuous curve $t\mapsto \lambda(t)\in T_{q(\cdot)}^*M\setminus\{0\}$ such that $\mathcal{H}(\lambda(t),u(t))$ maximizes $\mathcal{H}(\lambda(t),\cdot)$ for a.e. $t\in [0,T]$, and such that $\dot{\lambda}(t)=\overrightarrow{\mathcal{H}}(\lambda(t),u(t))$ a.e. on $[0,T]$. (For the precise definition of the Hamiltonian vector field $\overrightarrow{\mathcal{H}}$ and further details see Section~\ref{s:notations}.) 
		We call the triple $(q(\cdot),u(\cdot),\lambda(\cdot))$ an \emph{extremal triple}.
		In particular, the PMP reduces the problem of finding time-optimal trajectories  to the study of extremal ones.
	
	The kind of results we are interested in concern the regularity of time-optimal trajectories, even though our techniques handle in fact the broader class of 
extremal ones. Observe in any case that this is a hopeless task in full generality since, as proved by  Sussmann in \cite{Sussmann1986}, for any given measurable control $t\mapsto u(t)$, there exist a dynamical system of the form \eqref{eq:si} and an initial datum $q_0\in M$ for which the admissible trajectory driven by $u$ and starting at $q_0$ is time-optimal. 
	It makes then sense to look for better answers imposing some genericity conditions on $(f_0,f_1)$ (with respect to the Whitney topology on the space of pairs of smooth vector fields on $M$). The question we are then lead to tackle is the following: ``What kind of behavior can we expect for time-optimal trajectories of a generic system?'' Such a question corresponds to one of the open problems posed by A.~Agrachev in \cite{AAA-open-problems}.
	
	The problem of the regularity of extremal trajectories for control-affine systems of the form \eqref{eq:si} is known to be delicate. In his striking example, Fuller \cite{Fuller} exhibited a polynomial system of the kind studied here, in which controls associated with optimal trajectories 
have a converging sequence of isolated discontinuities.
	Since then the phenomenon of \emph{fast oscillations} (or \emph{chattering}) is also called the Fuller phenomenon, and his presence has important consequences for example on the study of optimal syntheses \cite{boscain-piccoli,caponigro-ghezzi-piccoli-trelat,marchal,piccoli2000regular,sussmann1990synthesis}. Another striking feature of this phenomenon is its stability: if the dimension of $M$ is sufficiently high, then chattering is structurally stable (i.e., it cannot be destroyed by a small perturbation of the initial system). The first result in this direction was presented in \cite[Theorem 0]{kupka1990ubiquity} starting from dimension $6$, but it was subsequently extensively explored in \cite{zelikin2012theory}. 
	It is however worth mentioning the fact that, to the best of our knowledge, none of these extremal trajectories have yet been proved to be time-optimal, nor it is known in lower dimensions (already in the $3$D case) whether or not the chattering appears for a generic choice of system \eqref{eq:si}. Finally, we remark that the 
	absence of Fuller phenomena for \eqref{eq:si} has been proved
	 in dimension $2$ for analytic systems and generic smooth systems \cite{lobry1970contro,piccoli1996regular,sussmann1982time,sussmann1987structure}. A first extensive investigation of the chattering phenomenon for multi-input affine-control systems has been presented in \cite{zelikin2012geometry}.
	
	\subsection{Fuller times along extremals trajectories} Many contributions have been provided to the description of 	
	the structure of optimal trajectories around a given point $q\in M$. The natural setting in which this problem is usually tackled is the study of all possible Lie bracket configurations between $f_0$ and $f_1$ at $q$ \cite{agrachev1990symplectic,AgrachevSigalotti,bressan1986,KS89,schattler1988local,schattler88altro,Sig05,sussmann1986envelopes}. This approach,  although very precise in its answers, has unfortunately the disadvantage of becoming 
	computationally 
	extremely  difficult 
	already for mildly degenerate situations in dimension $3$.
	
	\begin{defi}\label{defi:O}
	Given an admissible trajectory $q:[0,T]\to M$ of \eqref{eq:si}, 
	we denote by $O_q$ (or simply $O$ if no ambiguity is possible) the maximal open subset of $[0,T]$ such that there exists a 
	control $u:[0,T]\to [-1,1]$, associated with $q(\cdot)$, which is smooth on $O$. We also define $\Sigma_q$ (or $\Sigma$ if no ambiguity is possible) by
	\be\label{eq:Sigma} \Sigma=[0,T]\setminus O. \ee
	An \emph{arc} is a connected component of $O$. 
	An arc $\omega$ is said to be \emph{bang} if $u$ can be chosen so that $|u|\equiv 1$ along $\omega$, and \emph{singular} otherwise.
	Two arcs are \emph{concatenated} if they share one endpoint. The time-instant between two arcs is a \emph{switching time}.
	\end{defi}

The set $O_q$,  defined as above, depends only on the trajectory $q$ in the following sense: as long as $f_1(q(t))$ is different from zero, the control $u(t)$ is uniquely identified up to modification on a set of measure zero, while $u$ can be chosen arbitrarily on $\{t\mid f_1(q(t))=0\}$.

	\begin{defi}[Fuller Times]
		\label{defi:FT} Let $\Sigma_0$ be the set of isolated points in $\Sigma$ and define the \emph{Fuller times} as the elements of the set $\Sigma\setminus \Sigma_0$. 
		By recurrence, $\Sigma_k$ is defined as the set of isolated points of $\Sigma\setminus (\cup_{j=0}^{k-1}\Sigma_j)$. If $t\in\Sigma_k$ then $t$ is a \emph{Fuller time of order $k$}. We say that a Fuller time is of \emph{infinite order} if it belongs to
		\be
		\Sigma_\infty=\Sigma\setminus(\cup_{k\geq 1}\Sigma_k).
		\ee
	\end{defi} 
	
	The leading idea of this paper is to characterize the worst stable behavior for generic single-input systems of the form \eqref{eq:si}, in terms of the maximal order of its Fuller times. The heuristics behind our strategy is the following: thinking of points in $\Sigma\setminus\Sigma_0$ as ``accumulations of switchings'', points in $\Sigma\setminus(\Sigma_0\cup \Sigma_1)$ as ``accumulations of accumulations'' and so on, then if $t$ is a Fuller time of sufficiently high order, a large number of relations between $f_0(q(t))$ and $f_1(q(t))$ can be derived. 
	The existence of such a point $q(t)$ can then be ruled out by 
standard arguments based on Thom's transversality theorem (see, e.g., \cite[Proposition 19.1]{abraham-robbin}, which can be used in combination with \cite[\S 1.3.2]{goresky2012stratified} in order to the 
guarantee 
that the dense set of ``good'' systems can be taken open with respect to the Whitney $C^\infty$ topology on the space of vector fields). 
	The main result of this paper is the following.

	\begin{theorem}\label{t:main}
		Let $M$ be a $n$-dimensional smooth manifold. There exists 
		an open and dense set $\mc{V}\subset\mathrm{Vec}(M)\times\mathrm{Vec}(M)$ such that, 
		if the pair $(f_0,f_1)$ is in $\mc{V}$, 
		then for every extremal triple $(q(\cdot),u(\cdot),\lambda(\cdot))$ of the time-optimal control problem
		\be
		\dot q=f_0(q)+uf_1(q),\quad q\in M,\ u\in[-1,1],
		\ee
		the trajectory $q(\cdot)$ has at most Fuller times of order $(n-1)^2$, i.e.,
		$$\Sigma=\Sigma_0\cup \dots \cup \Sigma_{(n-1)^2},$$
		where $\Sigma$ and $\Sigma_j$ are defined as in Definition~\ref{defi:FT}.
	\end{theorem}	 

	\begin{remark}\label{rem:smooth}
		Since each $\Sigma_i$, for $i=1,\dots,(n-1)^2$, is discrete, as a consequence of Theorem~\ref{t:main} we deduce that the control $u(\cdot)$ associated with any extremal triple $(q(\cdot),u(\cdot),\lambda(\cdot))$ is smooth out of a finite union of discrete sets (in particular, out of a set of measure zero).
	\end{remark}
	
	As we already explained, deriving dependence relations directly on $f_0$ and $f_1$ is extremely complicated. The PMP naturally suggests to rather search for conditions in the cotangent space $T^*M$, where they are more easily characterizable,
and to subsequently project them down on the level of vector fields. On the other hand,  
the estimate $(n-1)^2$ on the maximal order of Fuller points obtained in this way 
is far from being optimal. 
The computation of the sharpest bound on the  order of Fuller points 
is still an open problem.

	\subsection{Structure of the paper} In Section~\ref{s:notations} we introduce the technical tools we need in the rest of the paper and 
	we present a brief survey of related results.
	 Section~\ref{sec:Start} is the starting point of the novel contributions of the paper: we prove that at Fuller times of order larger than zero, i.e., for $t\in \Sigma\setminus\Sigma_0$, in addition to the conditions $\langle \lambda(t), f_1(q(t))\rangle=\langle \lambda(t), [f_0,f_1](q(t))\rangle=0$, one also has that either $\langle\lambda(t), [f_0+f_1,[f_0,f_1]](q(t))\rangle=0$ or  $\langle\lambda(t), [f_0-f_1,[f_0,f_1]](q(t))\rangle=0$. 
The computations leading to this result do not require any genericity assumption.
Section~\ref{sec:HO}, which constitutes the technical core of this work, explains how to derive new conditions at each accumulation step and how to prove their independence. Section~\ref{sec:proof} concludes the proof of Theorem \ref{t:main} and, finally, in Section~\ref{sec:3}, the case of time-optimal trajectories on three dimensional manifolds is analyzed in greater detail.

\subsection*{Acknowledgements} 
The authors have been supported by the ANR SRGI (reference ANR-15-CE40-0018) 
and 
by a public grant as part of the {\it Investissement d'avenir} project, reference ANR-11-LABX-0056-LMH, LabEx LMH, in a joint call with {\it Programme Gaspard Monge en Optimisation et Recherche Op\'erationnelle}. The authors also wish to thank the anonymous referee for the careful revision of the paper, and the detailed comments that let us significantly improve the quality of our exposition.

	\section{Previous results and consequences of Theorem~\ref{t:main}}\label{s:notations}
\subsection{Notations}

Let us introduce some technical notions which will be extensively used throughout the rest of the paper. Consider the cotangent space $T^*M$ of $M$, endowed with the canonical symplectic form $\sigma$. For any Hamiltonian function $p:T^*M\to \R$, its Hamiltonian lift $\overrightarrow{p}\in\mathrm{Vec}(T^*M)$ is defined using the relation $$\sigma_\lambda(\cdot,\overrightarrow{p})=\langle d_\lambda p,\cdot\rangle.$$

For all $T>0$ and $q_0\in M$ we define the \emph{attainable set} from $q_0$ at time $T$ as
\be\label{eq:att}
	A(T,q_0)=\{q(t)\mid q:[0,T]\to M\textrm{ is an admissible trajectory of \eqref{eq:si} such that }q(0)=q_0\}. 
\ee

The precise content of the PMP, already mentioned at the beginning of Section~\ref{sec:intro}, is then recalled below (see \cite{agrachevbook,PMP}). 

\begin{thmi}[PMP]
	Let $q:[0,T]\to M$ be an admissible trajectory of \eqref{eq:si}, associated with a 
	control $u(\cdot)$, such that $q(T)\in\partial A(T,q_0)$. Then there exists
	$\lambda:[0,T]\to T^*M$ absolutely continuous
	such that
	$(q(\cdot), u(\cdot),\lambda(\cdot))$ is an \emph{extremal triple}, i.e., in terms of the control-dependent Hamiltonian $\mathcal{H}$ introduced in \eqref{eq:maxH},
	\begin{align}
	&\lambda(t)\in T^*_{q(t)}M\setminus \{0\},\quad \forall t\in [0,T],\\
	\label{eq:maxPMP}
	&\mathcal{H}(\lambda(t),u(t))=\max_{v\in[-1,1]}\mathcal{H}(\lambda(t),v),\quad \mathrm{for\ a.e.\ }t\in[0,T],\\
	&\dot{\lambda}(t)=\overrightarrow{\mathcal{H}}
	(\lambda(t),u(t)),\quad \mathrm{for\ a.e.\ }t\in[0,T].\label{eq:dynPMP}
	\end{align}
\end{thmi}

Let $(q(\cdot), u(\cdot),\lambda(\cdot))$ be an extremal triple. The curve $q(\cdot)$ is in particular said to be an \emph{extremal trajectory}.
We associate with $(q(\cdot), u(\cdot),\lambda(\cdot))$ the \emph{switching function}
	\be
		h_1(t)=\langle\lambda(t),f_1(q(t))\rangle.	
	\ee
	Differentiating a.e. on $[0,T]$, it follows from \eqref{eq:dynPMP} that  for every smooth vector field $X$ on $M$ 
	\be
		\frac{d}{dt}\langle\lambda(t),X(q(t))\rangle=\langle \lambda(t),[f_0+u(t)f_1,X](q(t))\rangle, \qquad \mbox{for a.e. }t\in[0,T].
	\ee
	In particular, $h_1$ is of class $C^1$ and, setting 
	\be
		h_{01}(t)=\langle \lambda(t),[f_0,f_1](q(t))\rangle,\quad \forall t\in[0,T],
	\ee
we have	$\dot{h}_1(t)=h_{01}(t)$ for every $t\in [0,T]$.
\begin{remark}\label{rmk:sign}
	The maximality condition \eqref{eq:maxPMP} implies that $$\mathcal{H}(\lambda(t),u(t))=\langle \lambda(t),f_0(q(t))\rangle+\max_{v\in[-1,1]}v\langle\lambda(t),f_1(q(t))\rangle= \langle \lambda(t),f_0(q(t))\rangle+
	|\langle\lambda(t),f_1(q(t))\rangle|.$$
	In particular, $u(t)=\mathrm{sgn}(h_1(t))\in\{-1,+1\}$ whenever $h_1(t)\neq 0$.
\end{remark}	
	Repeated differentiation shows that $h_1$ is smooth when the control is. In particular, in terms of the set $O$ introduced in Definition~\ref{defi:O}, $h_1|_O\in C^\infty(O)$. 
	
A folklore result on bang and singular arcs 
is the following. 	Recall that, for every $f\in {\rm Vec}(M)$,  ${\rm ad}_f :{\rm Vec}(M)\to {\rm Vec}(M)$ denotes the adjoint action defined by 
${\rm ad}_f g=[f,g]$. 
		\begin{prop}
Assume that ${\rm span}\{({\rm ad}^k_{f_0+ f_1}f_1)(q)\mid k\in \N\}=T_qM$ and 
${\rm span}\{({\rm ad}^k_{f_0- f_1}f_1)(q)\mid k\in \N\}=T_qM$
for every $q\in M$.
Fix an extremal triple $(q(\cdot),u(\cdot),\lambda(\cdot))$ and an arc $\omega\subset O_q$. Then, either 
  $h_1(t)= 0$ for at most finitely many $t\in \omega$ and the arc is bang, or  
   $h_1\equiv 0$ on $\omega$ and the arc is singular.
	\end{prop}
\begin{proof}
Let us set 
$Z=\{\tau\in\omega\mid h_1(\tau)=0\}$ and $F^\pm= \{\tau\in\omega\mid \pm h_1(\tau)>0\}$. Assume by contradiction that $Z$ has infinitely many points and that it is different from $\omega$.
We have from Remark~\ref{rmk:sign}  that, up to modifying $u$ on a set of measure zero,  $u\equiv 1$ on $F^+$  and $u\equiv -1$ on $F^-$.
If $Z$ has measure $0$, then, by continuity of $u|_\omega$,  $u\equiv 1$ or $u\equiv -1$ on $\omega$.
In particular,  
$h_1^{(k)}(t)=\langle\lambda(t),({\rm ad}^k_{f_0+ f_1}f_1)(q(t))\rangle$ or 
$h_1^{(k)}(t)=\langle\lambda(t),({\rm ad}^k_{f_0- f_1}f_1)(q(t))\rangle$
on $\omega$. 
Since between any two vanishing points for $h_1^{(k-1)}$ there is a vanishing point for $h_1^{(k)}$,
we deduce that at every cluster point $t\in \bar \omega$ for $Z$ (i.e., the limit of infinitely many distinct points in $Z$),  
$\lambda(t)$ annihilates either $({\rm ad}^k_{f_0+ f_1}f_1)(q(t))$ for every $k\in \N$ or $({\rm ad}^k_{f_0- f_1}f_1)(q(t))$ for every $k\in \N$, leading to a contradiction.

In the case where the measure of $Z$ is positive, there exists $t\in \omega$ which is both a cluster point for $Z$ and for either $F^+$ or $F^-$.
By continuity of $h_1^{(k)}|_\omega$ for every $k\in\N$, 
we deduce that either
$h_1^{(k)}(t)=\langle\lambda(t),({\rm ad}^k_{f_0+ f_1}f_1)(q(t))\rangle$ 
for every $k\in \N$ or 
$h_1^{(k)}(t)=\langle\lambda(t),({\rm ad}^k_{f_0- f_1}f_1)(q(t))\rangle$ for every $k\in \N$ and we conclude as above.
\end{proof}

Notice that the assumption that  
${\rm span}\{({\rm ad}^k_{f_0+ f_1}f_1)(q)\mid k\in \N\}=T_qM$ and 
${\rm span}\{({\rm ad}^k_{f_0- f_1}f_1)(q)\mid k\in \N\}=T_qM$
for every $q\in M$ holds true  generically with respect to $(f_0,f_1)\in\mathrm{Vec}(M)^2$.
From now on the term \emph{generic} is used to express that a property of the pair of vector fields $(f_0,f_1)$ holds true on an open and dense subset of $\mathrm{Vec}(M)\times \mathrm{Vec}(M)$.

	\begin{defi}
	\label{def:word}
	Let $\mathfrak{A}$ be the alphabet containing the letters $\{+,-,0,1\}$, and let $I=(i_1\cdots i_d)\in\mathfrak{A}^d$ be a word of length $d$ in $\mathfrak{A}$. 
	Then 
	we employ the shorthand notation 
	$$f_I=[f_{i_1},\dots,[f_{i_{d-1}},f_{i_d}]\dots],$$
	with the convention that $f_{\pm}=f_0\pm f_1$.
	Moreover, given an extremal triple $(q(\cdot),u(\cdot),\lambda(\cdot))$ on $[0,T]$,
	we set 
	\begin{align*} 
	h_{I}(t)&=
	\langle\lambda(t),f_I(q(t))\rangle,\qquad t\in[0,T].
	\end{align*}
\end{defi}

\subsection{Previous results}
	
Sussmann proved in \cite{Sussmann1986} that for every $T>0$ and every control $u\in L^\infty([0,T],[-1,1])$ there exists a control system of the type \eqref{eq:si} and an initial datum $q_0$ such that the trajectory starting at $q_0$ and corresponding to $u(\cdot)$ is time-optimal. 
In generic situations, however, some further regularity can be expected, as  recalled in the following three results.

\begin{theorem}[{\cite[Theorem 0]{BonnardKupka}, \cite[Theorem 2.6]{CJT-SIAM}}]\label{t:BK}
Generically with respect to $(f_0,f_1)\in \mathrm{Vec}(M)^2$,  
for any extremal triple $(q(\cdot),u(\cdot),\lambda(\cdot))$ on $[0,T]$ 
such that $h_1|_{[0,T]}\equiv 0$, 
 the set $\Omega=\{t\in [0,T]\mid h_{101}(t)\ne 0\}$ is of full measure in $[0,T]$ and  
 $u(t)=-h_{001}(t)/h_{101}(t)$ almost everywhere on $\Omega$.
\end{theorem}

\begin{theorem}[{\cite[Proposition 1]{Agrachev1995}}]\label{t:AAA}
	Let $I(f_1)\subset \mathrm{Lie}(f_0,f_1)$ denote the ideal generated by $f_1$. If $I_q(f_1)=T_qM$ for every $q\in M$, then, for every extremal trajectory $q:[0,T]\to M$, 
the set $O_q$ is open and dense in $[0,T]$. 	
\end{theorem}

\begin{theorem}[{\cite[Proposition 2]{AgrachevSigalotti}}]\label{t:AS}
Assume that ${\rm span}\{({\rm ad}^k_{f_0+ f_1}f_1)(q)\mid k\in \N\}=T_qM$ and 
${\rm span}\{({\rm ad}^k_{f_0- f_1}f_1)(q)\mid k\in \N\}=T_qM$
for every $q\in M$.
Consider an extremal trajectory $q:[0,T]\to M$ 
such that 
the union of all bang arcs is open and dense in $[0,T]$.
Then either $\Sigma=\Sigma_0$ or there exists an infinite sequence of concatenated bang arcs. 
\end{theorem}

Theorem~\ref{t:main} 
can be seen as an extension of Theorem~\ref{t:AAA} 
in the sense that it guarantees that, generically with respect to $(f_0,f_1)$, the open set $O_q$ is not only dense but also of countable complement and hence of full measure in $[0,T]$ (see Remark~\ref{rem:smooth}).	
A similar observation can be done for Theorem~\ref{t:AS}, which is generalized by Theorem~\ref{t:main} as follows: generically, for every $k\ge 0$, either $\Sigma=\cup_{j=0}^k \Sigma_j$ or there exists
a subinterval $I$ of $[0,T]$ such that $I\cap \Sigma_k$ is 
 a 
 converging 
 sequence.

  Concerning Theorem~\ref{t:BK}, we can strengthen its conclusion 
 as stated in Corollary~\ref{cor:countableinBK} below. The corollary is a direct consequence of Proposition~\ref{p:intermezzo}, which is a step of the
 proof of Theorem~\ref{t:main} contained in 
 Section~\ref{sec:proof}.
 \begin{corollary}\label{cor:countableinBK}
Generically with respect to the pair $(f_0,f_1)\in \mathrm{Vec}(M)^2$,  
for any extremal triple $(q(\cdot),u(\cdot),\lambda(\cdot))$ on $[0,T]$ 
such that $h_1|_{[0,T]}\equiv 0$, 
 the set $\Omega=\{t\in [0,T]\mid h_{101}(t)\ne 0\}$ has countable complement in $[0,T]$ and  
 $u(t)=-h_{001}(t)/h_{101}(t)$ almost everywhere on $\Omega$.
\end{corollary}

\subsection{Chattering and singular extremals} Classical instances of the chattering phenomenon occur when trying to join singular and bang arcs along time-optimal trajectories of control systems as in \eqref{eq:si}. Legendre condition \cite[Theorem 20.16]{agrachevbook} holds along singular extremal triples, and imposes the inequality $h_{101}(t)\geq 0$. If the inequality is strict, then the control $u(t)$ is characterized as in Theorem~\ref{t:BK}, but there are significant examples of mechanical problems in which the third bracket $f_{101}$ vanishes identically (e.g. Dubin's car with acceleration \cite[Section 20.6]{agrachevbook}). This case has been intensively studied in \cite{zelikin2012theory}, and the situation that forces the chattering can be essentially summarized as follows.
\begin{theorem}[{\cite[Proposition 20.23]{agrachevbook}}] Assume that the vector fields $f_0$ and $f_1$ satisfy the identity $f_{101}\equiv 0$. Let $q:[0,T]\to M$ be a  time-optimal trajectory of  system~\eqref{eq:si} which is the projection of a unique (up to a scalar factor)
curve $\lambda:[0,T]\to T^*M$ such that 
 $(q(\cdot),u(\cdot),\lambda(\cdot))$ is an extremal triple. Assume moreover that $h_{10001}(t)\neq 0$ on $[0,T]$. Then $q(\cdot)$ cannot contain a singular arc concatenated with a bang arc.
\end{theorem}		  
In particular, under the hypotheses of the theorem, the only possibility for an optimal trajectory to exit a singular arc is through chattering.

\section{Annihilation conditions at Fuller times of an extremal trajectory}\label{sec:Start}

Let us fix an extremal triple $(q(\cdot),u(\cdot),\lambda(\cdot))$ on $[0,T]$. 
The goal of this section is to prove some useful annihilation conditions of functions of the form $h_I$, with $I$ a word in $\mathfrak{A}$ (compare with Definition~\ref{def:word}), at Fuller times, i.e., on $\Sigma\setminus\Sigma_0$.   

Since $h_1$ is (absolutely) continuous and  $u(t)=\mathrm{sgn}(h_1(t))$ for almost every $t$ such that $h_1(t)\neq 0$, then
$$h_1\big|_{\Sigma}\equiv 0.$$
Moreover, between two zeroes of $h_1$, $h_1^{(1)}=h_{01}$ has at least one zero, which yields 
$$h_{01}\big|_{\Sigma\setminus \Sigma_0}\equiv 0.$$

The following proposition states that both 	$h_{001}$ and $h_{101}$ vanish at every 
$t\in\Sigma$ which is at positive distance from 
$\{t\mid h_1(t)\ne 0\}$. 

	\begin{prop}\label{p:SS}
		Let $t\in\Sigma$ be such that 
		$h_1$ is identically equal to zero on 
	a neighborhood 
	of $t$.
 Then $h_{101}(t)=h_{001}(t)=0$.
	\end{prop}
	
	\begin{proof}
Let $V$ be a neighborhood of $t$ such that 
$h_1|_V\equiv 0$. Therefore, the same is true for $h_{01}|_V$ and
\begin{equation}\label{h001+uh101=0}
	h_{001}(\tau)+u(\tau)h_{101}(\tau)=0 \quad \mbox{for almost every $\tau\in V$.} 
\end{equation}	
	
Let us first prove that $h_{101}(t)=0$. 
By contradiction and up to reducing $V$, we have that $h_{101}(\tau)\neq 0$ for every $\tau\in V$. 
By \eqref{h001+uh101=0}, moreover, $u(\tau)=-\frac{h_{001}(\tau)}{h_{101}(\tau)}$ for almost every $\tau\in V$.

Notice that the differential system generated by the smooth autonomous Hamiltonian
		\be
			\label{eq:HamSing}
			H(p)=\langle p, f_0(\pi(p))\rangle-\frac{\langle p, f_{001}(\pi(p))\rangle}{\langle p, f_{101}(\pi(p))\rangle}\langle p, f_1(\pi(p))\rangle,
		\ee
is well-defined on $\{p\in T^*M\mid \langle p, f_{101}(\pi(p))\rangle\ne 0\}$	and all its trajectories are smooth.
Since, moreover, the absolutely continuous curve $(\lambda(t),q(t))$ 
satisfies $\dot p=\ovr{H}(p)$  almost everywhere on $V$, we deduce that $V\ni t\mapsto (\lambda(t),q(t))$  is a solution of the Hamiltonian system 
generated by $H$ and that the control $u$ is smooth on $V$, contradicting the fact that $t\in\Sigma$. 

We conclude by showing that also $h_{001}(t)=0$. 
Following \eqref{h001+uh101=0}, 
we have
$$ |h_{001}(\tau)|=|u(\tau)| |{h_{101}(\tau)}|\le |{h_{101}(\tau)}|\quad \mbox{for almost every $\tau\in V$}$$
and then we conclude by continuity of $h_{101}$ and $h_{001}$.
	\end{proof}

	\begin{prop}\label{p:BBB}
		Assume that there exists 
an infinite sequence of concatenated bang arcs converging to $t\in [0,T]$. 
Then either $h_{+01}(t)=0$ or $h_{-01}(t)=0$. 
	\end{prop}
	
	\begin{proof}
	First notice that $t\in \Sigma\setminus \Sigma_0$.
		Assume by contradiction that neither $h_{+01}(t)$ nor $h_{-01}(t)$ is equal to zero. 
		Consider a neighborhood $I$ of $t$ in  $[0,T]$ (with respect to the topology induced by $\R$) such 
		that
		\be
			\frac{1}{C}\leq |h_{+01}(s)|,|h_{-01}(s)|\leq C,\quad \forall\,s\in I,
		\ee
for some positive constant $C>0$.

		By assumption, there exists a sequence of concatenated bang arcs in $I$, whose lengths we denote by  $\{\sigma_i\}_{i\in \N}\cup \{\tau_i\}_{i\in\N}\subset(0,+\infty)$, 
		with the agreement that $u\equiv 1$ (respectively, $u\equiv-1$) on the intervals of length $\sigma_i$ (respectively, $\tau_i$)
		and that the  arc of length $\sigma_i$ is concatenated with the  arc of length $\tau_i$, which is concatenated with the arc of length $\sigma_{i+1}$ and so on.
		Without loss of generality,  the bang arcs converge towards $t$ from the left, so that we can further assume that the  arc of length  $\sigma_i$ is concatenated at its right 
		with the 
		 arc of length $\tau_i$
		 (see Figure~\ref{f:4arcs}).

		\begin{figure}[h!]
			\includegraphics[scale=.8]{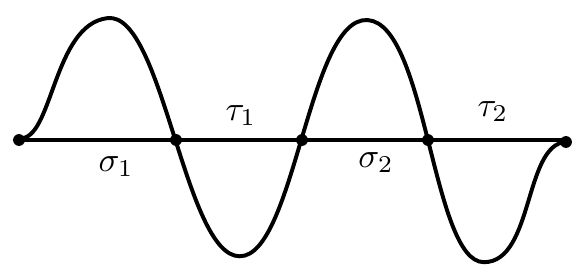}
			\caption{A concatenation of bang arcs\label{f:4arcs}}
		\end{figure}

By convention, let $0$ be the starting time of the sequence in Figure~\ref{f:4arcs}.	
		Taylor's formula yields that
		\be
			\label{eq:firstorder}
			\sigma_1=-\frac{2h_{01}(0)}{h_{+01}(0)}+O(\sigma_1^2),
		\ee
		where the notation $O(\sigma_1^2)$ has the following meaning: using an analogous Taylor expansion for each positive bang arc of length $\sigma_k$, we obtain a reminder $\rho_k$ 
		such that
$\frac{\rho_k^2}{\sigma_k^2}$ is uniformly bounded. 
		We deduce from \eqref{eq:firstorder} that $|h_{01}(0)|\approx \sigma_1$, where this 
		notation is used to  indicate that 
		\[ \frac1c\le \frac{|h_{01}(0)|}{\sigma_1}\le c \]
		for some constant $c>0$.
Moreover, from the expansion $h_{01}(\sigma_1)=h_{01}(0)+\sigma_1h_{+01}(0)+O(\sigma_1^2)$, we get
		\be
			\label{eq:secder}
			h_{01}(\sigma_1)=-h_{01}(0)+O(\sigma_1^2).
		\ee
	
		Combining these two relations we obtain \be\label{eq:comp1}\tau_1\approx|h_{01}(\sigma_1)|\approx \sigma_1.\ee
	
		The same computations also imply  that
		\be\label{eq:series}\sigma_2\approx h_{01}(\sigma_1+\tau_1)=-h_{01}(\sigma_1)+O(\tau_1^2)=h_{01}(0)+O(\sigma_1^2).\ee
		In particular, the sequence $\sigma_i$ satisfies the relation $\sigma_{i+1}=\sigma_i+O(\sigma_i^2)$. The contradiction is then a consequence of  Lemma~\ref{l:below} below.
	\end{proof}
	
	\begin{lemma}\label{l:below}
		Let $\{t_i\}_{i\in\N}$ be a sequence of positive numbers satisfying the relation
		\be t_{i+1}=t_i+O(t_i^2).\ee 
		Then, $\sum_{i=1}^\infty t_i=+\infty$.
	\end{lemma}
	
	\begin{proof}
	Let $c>0$ be such that
		\be\label{eq:constant}t_{i+1}\geq t_i(1-ct_i), \quad \forall\;i\in\N.\ee
		Assume by contradiction that $\sum_{i=1}^\infty t_i<+\infty$. In particular,  $t_i\to 0$.

Up to discarding the first terms of the sequence $\{t_i\}_{i\in\N}$, we can assume that $1-ct_i>0$ for all $i\in\N$.	
Iterating \eqref{eq:constant} we deduce that 		
		\be\label{eq:product}t_{i+1}\geq t_1 \prod_{j=1}^i (1-ct_j), \quad \forall\;i\in\N.\ee
Hence, for every $i\in \N$, 
\[\log t_{i+1}\geq \log t_1+ \sum_{j=1}^i \log(1-ct_j)\ge \log t_1-  c'\sum_{j=1}^i t_j,\]
where $c'>0$ is such that $\log(1-ct_j)\ge - c' t_j$ for all $j\in \N$. The contradiction comes by noticing that the 
left-hand side goes to $-\infty$ as $i\to \infty$, while the right-hand side stays uniformly bounded.
	\end{proof}

We say that an arc is \emph{bi-concatenated} if it is concatenated both at its right and at its left with other arcs. 	 
	
	\begin{prop}\label{p:NC}
		Let $I$ be a bang arc compactly contained in $(0,T)$ and which is not bi-concatenated. Then there exists $t\in\bar{I}$ such that either $h_{+01}(t)=0$ or $h_{-01}(t)=0$.
	\end{prop}

	\begin{proof}
	Without loss of generality, assume  that $u\equiv 1$ on $I=(t_1,t_2)$  
		and that 
$I$ is not concatenated with any other arc at $t_2$. In particular, $t_2$ is a cluster point  for $\Sigma\cap (t_2,T]$. 
If $h_1\equiv 0$ on a right neighborhood of $t_2$, then the conclusion follows from
Proposition~\ref{p:SS} and the continuity of $h_{+01}$ and $h_{-01}$. 		

We can then assume that 
 there exists a sequence of times 
converging from above to $t_2$ and at which $h_1$ is not zero. 
Then, necessarily, 
there exist a sequence of arcs $I_n$ converging to $t_2$. 
Pick, for every $n\in \N$ a time $\tau_n\in I_n$ 	such that $h_{01}(\tau_n)=0$. By construction, the sequence $(\tau_k)_{k\in\N}$ converges to $t_2$ and, by continuity, we deduce that also $h_{01}(t_2)=0$.

Since $h_1(t_1)=h_1(t_2)=0$, then by the mean value theorem $h_{01}$ vanishes at an interior point of $I$, and this in turns implies that $\frac{d}{dt}h_{01}|_{I}=h_{+01}|_I$ also vanishes somewhere on $I$. 
	\end{proof}

The main result of the section is the following theorem. 
	
	 \begin{theorem}\label{thm:start}
Let $t\in\Sigma\setminus\Sigma_0$. 
		Then $h_1(t)=h_{01}(t)=0$ and, in addition, either $h_{+01}(t)=0$ or $h_{-01}(t)=0$. 
	 \end{theorem}
	\begin{proof}
	We already noticed that $h_1$ vanishes on $\Sigma$ and $h_{01}$ on $\Sigma\setminus\Sigma_0$. We are going to prove the theorem by showing that there exists a sequence of points converging to $t$ at which either $h_{+01}$ or 
$h_{-01}$ vanishes.

Since $t\not\in \Sigma_0$ and 
thanks to Proposition~\ref{p:SS}, we can assume without loss of generality that
$h_1$ does not vanish identically on a neighborhood of $t$. Hence,  
there exists a sequence $(\tau_n)_{n\in \N}\subset[0,T]$ converging to $t$ such that $h_1(\tau_n)\ne 0$ for every $n\in\N$. Each $\tau_n$ is contained in an arc $\omega_n$.  If the arc is singular, then 
it contains a nonempty subinterval on which $h_1\equiv 0$. Since moreover $h_1$ has either a positive maximum or a negative minimum on $\omega_n$, we deduce 
that there exists an inflection point of $h_1$  on $\omega_n$ at which $h_{+01}$ or 
$h_{-01}$ vanishes. 

We can then assume without loss of generality  that $\omega_n$ is a bang arc for every $n\in\N$. Let us consider the maximal concatenation of bang arcs from $\omega_n$ towards $t$. Three possibilities occur: (i) the concatenation is infinite
 and converges to a point between $\tau_n$ and $t$, 
(ii) the concatenation stops with a bang arc which is not bi-concatenated, and (iii) the concatenation stops with a bang arc concatenated with a singular one. 
In each of the three cases, we prove that there exists a point between $\omega_n$ and $t$ at which either $h_{+01}$ or 
$h_{-01}$ vanishes. 
In cases (i) and (ii) the conclusion follows from Propositions~\ref{p:BBB} and \ref{p:NC} respectively.
In the case of a bang arc concatenated with a singular one, either $h_1$ 
does not vanish everywhere on the singular arc, and we  
deduce as above that there exists an inflection point of $h_1$  on the singular arc at which $h_{+01}$ or 
$h_{-01}$ vanishes, or $h_{01}=0$ at the junction of the two arcs and then the bang arc contains an inflection point of $h_1$  at which $h_{+01}$ or 
$h_{-01}$ vanishes. This concludes the analysis in case (iii) and hence the proof of the theorem.
	\end{proof}

\section{High-order Fuller points and genericity results}\label{sec:HO}

In this section we look at the new dependence conditions appearing for accumulations of Fuller points of order higher than one. 
We start by introducing some useful notation.

\begin{remark}\label{r:unnumbered}
	For any given word $J=(j_1,\dots,j_r)\in\mathfrak{A}^r$, with $r\geq 3$, $j_{r-1}=0$, $j_r=1$, and at least
	one $j_k$ in $\{+,-\}$,
	 an easy inductive argument proves that, with the notations of Definition \ref{def:word}, we can 
	decompose $f_J$ as 
	\be
	f_J=f_{J_1}+\dotsb+f_{J_l},
	\ee
	where $J_1,\dotsc, J_l$ are all words of length $r$ written only with letters in $\{0,1\}$, ending with the string $(01)$ and such that, if $|J_i|_a$ counts the number of occurrences of the letter $a$ in $J_i$, then
	\be
	|J_1|_0=\max_{i=1,\dots,l}|J_i|_0,\quad\textrm{and}\quad 	|J_2|_1=\max_{i=1,\dots,l}|J_i|_1.
	\ee
	Moreover, $J_1$ and $J_2$ are uniquely determined by this requirement.
\end{remark}

\begin{defi}\label{def:relations}
Let $N\in\N$. A function	$S:T^*M\times J^NM\times J^NM\to\R$ is said to be a \emph{simple relation of degree} $d\le N$ 
if there exists  a word  $I\in\mathfrak{A}^d$ of length $d$ such that
$S=S_I$, where 
\begin{equation}\label{eq:simple_relation}
S_I(\lambda,j^N_q(f_0),j^N_q(f_1))=\langle\lambda, f_I(q)\rangle,\quad q=\pi(\lambda).
\end{equation}
Similarly,  
	we call $Q:T^*M\times J^NM\times J^NM\to\R$ a \emph{polynomial relation} if there exist 
	$l,d_1,\dots, d_l\in\N\setminus\{0\}$ and words $I_1\in\mathfrak{A}^{d_1},\dots, I_l\in\mathfrak{A}^{d_l}$ such that
	\begin{equation}
	\label{eq:polynomial_relation}
		Q(\lambda,j^N_q(f_0),j^N_q(f_1))\in\R[S_{I_1}(\lambda,j^N_q(f_0),j^N_q(f_1)),\dots,S_{I_l}(\lambda,j^N_q(f_0),j^N_q(f_1))].
	\end{equation}
	Moreover, we set $\mathrm{deg}(Q)=\max\{d_1,\dots,d_l
	\}$.
	
	Finally, given two simple relations $S_I,S_J$, with a slight abuse of notation we say that the Poisson bracket $\{S_I,S_J\}$ between $S_I$ and $S_J$ is the simple relation $S_{IJ}$, where $IJ$ is defined by concatenation of words. We extend the Poisson bracket notation to polynomial relations by linearity and the Leibnitz rule.  
\end{defi}

In the following two lemmas we show how to derive new algebraic conditions on the jets of the vector fields $f_0$ and $f_1$ when increasing the order of the Fuller point. 

\begin{lemma}\label{lemma:Jets}
	Let $l,d_1,\dots, d_l\in\N\setminus\{0\}$ and consider $l$ words $I_1\in\mathfrak{A}^{d_1},\dots,I_l\in\mathfrak{A}^{d_l}$ with $d_j<d_l$ for every $j<l$ and 
	$I_l=(+I_{l-1})$, where we denote by $(+I_{l-1})$ the concatenation of the letter $+$ and the word $I_{l-1}$. 
	Fix an integer $N> d_l$ and  consider the family of simple relations $S_j=S_{I_j}$, $1\leq j\leq l$, using the notation introduced in \eqref{eq:simple_relation}. 
	Define the set $\mathcal{B}
	\subset T
	^*M\times J
	^NM\times J
	^NM$ by 
	\begin{align*}
		\mathcal{B}
		=\bigg\{(\lambda, j^N_q(f_0),j_q^N(f_1))\mid q=\pi(\lambda),\;&(f_0,f_1)\in {\rm Vec}(M)^2,\\ &S_1(\lambda, j^N_q(f_0),j_q^N(f_1))=\dots=S_l(\lambda, j^N_q(f_0),j_q^N(f_1))=0\bigg\}.
	\end{align*}
	If 
	 $(q(\cdot), u(\cdot),\lambda(\cdot))$ is an extremal triple on $[0,T]$ for the time-optimal control problem \eqref{eq:si} associated with the pair $(f_0,f_1)$,
and if the sequence $\{t_i\}_{i\in\N}\subset[0,T]$ is  such that
	\begin{itemize}
		\item [i)] $(\lambda(t_i), j^N_{q(t_i)}(f_0),j_{q(t_i)}^N(f_1))\in\mathcal{B}
		$ for every $i\in\N$,
		\item [ii)] there exists $t_\infty=\lim_{i\to\infty}t_i$,
	\end{itemize}
	then there exists a further simple relation
	$$S_{l+1}\in\left\{S_{(-I_{l-1})},S_{(-I_{l})},S_{(+I_{l})} \right\}$$
	such that 
	$$(\lambda(t_\infty), j^N_{q(t_\infty)}(f_0),j_{q(t_\infty)}^N(f_1))\in\mathcal{B}
	\cap\{S_{l+1}=0\}.$$
	Finally, defining  for every $q\in M$ the set $\mathcal{B}'_q\subset T^*_qM\times J_q^NM\times J_q^NM$ by 
\begin{align*}
\mathcal{B}'_q
	=\bigg\{(\lambda, j^N_q(g_0),j_q^N(g_1))\mid \lambda&\in T^*_q M\setminus\{0\},\;(g_0,g_1)\in {\rm Vec}(M)^2,\;g_0(q)\wedge g_1(q)\neq 0,\\
&	S_1(\lambda, j^N_q(g_0),j_q^N(g_1))=\dots=S_l(\lambda, j^N_q(g_0),j_q^N(g_1))=0\bigg\},
\end{align*}
if the codimension of $\mathcal{B}'_q$ in $T_q^*M\times J_q^NM\times J_q^NM$ 
is equal to $l$, then 
	$$\mathrm{codim}_{T_q^*M\times J_q^NM\times J_q^NM}(\mathcal{B}'_q\cap\{S_{l+1}=0\})=l+1.$$
\end{lemma}

\begin{proof}
	Let $(q(\cdot),u(\cdot),\lambda(\cdot))$ be an extremal triple defined on $[0,T]$ and $\{t_i\}_{i\in\N}\subset [0,T]$ be a sequence of points satisfying i) and ii) in the statement. Then, since for every word $J\in\{I_1,\dots, I_l\}$ we have that $h_{J}(t_i)=\langle \lambda(t_i),f_{J}(q(t_i))\rangle$ vanishes for every $i\in\N$, by continuity the same is also true for $h_{J}(t_\infty)$, which implies that the point $(\lambda(t_\infty), j^N_{q(t_\infty)}(f_0),j_{q(t_\infty)}^N(f_1))$ belongs to $\mathcal{B}
	$.
	
	Now, up to the choice of a suitable subsequence of $\{t_i\}_{i\in\N}$, we infer the identity
	\begin{align}\label{eq:integral_mean}
		0&=\lim_{i\to\infty}\frac{h_J(t_\infty)-h_J(t_i)}{t_\infty-t_i}=\lim_{i\to\infty}\frac{1}{t_\infty-t_i}\int_{t_i}^{t_\infty}(h_{0J}(\tau)+u(\tau)h_{1J}(\tau))d\tau\\
		&=h_{0J}(t_\infty)+\bar{u}h_{1J}(t_\infty),\quad \bar{u}=\lim_{i\to\infty}\frac{1}{t_\infty-t_i}\int_{t_i}^{t_\infty}u(\tau)d\tau\;\in [-1,1],
	\end{align}
	which is 
	valid for every $J\in\{I_1,\dots,I_l\}$. The first of our claims is then proved. Indeed, if $\bar{u}=\pm1$ we use \eqref{eq:integral_mean} with $J=I_l$ to deduce that
	$$
	\langle \lambda(t_\infty), f_{(\pm I_l)}(q(t_\infty)) \rangle=0,$$ so that $S_{l+1}$ is in the form $S_{(\pm I_l)}$, and we are done. If, on the other hand, $\bar{u}\in(-1,1)$ we apply \eqref{eq:integral_mean} with $J=I_{l-1}$, and we deduce that
	$$S_{(\bar{u},l-1)}(\lambda(t_\infty),j^N_{q(t_\infty)}(f_0),j_{q(t_\infty)}^N(f_1)):=\langle \lambda(t_\infty), f_{(0I_{l-1})}(q(t_\infty))+\bar{u}f_{(1I_{l-1})}(q(t_\infty))\rangle=0.$$ 
	The combination of the relations $S_{(\bar{u},l-1)}=S_l=0$ at the point $(\lambda(t_\infty),j^N_{q(t_\infty)}(f_0),j_{q(t_\infty)}^N(f_1))$ yields
	$$\langle \lambda(t_\infty),f_{(0I_{l-1})}(q(t_\infty))\rangle=0\quad \mathrm{ and }\quad\langle\lambda(t_\infty),f_{(1I_{l-1})}(q(t_\infty))\rangle=0,$$
	which in turn implies that
	$$
	\langle \lambda(t_\infty), f_{(- I_{l-1})}(q(t_\infty))\rangle=0,$$ so that
	we conclude by taking
	 $S_{l+1}=S_{(-I_{l-1})}$.
	
	To prove the second claim of the statement, it is not restrictive to work within a coordinate neighborhood $(U,x)\subset \R^n$ centered at the origin (identified with $q$), the whole argument being local. Then $g_i(x)=\sum_{j=1}^n\alpha_i^j(x)\partial_{x_j}$ on $U$, for $i=0,1$. On $J^N_0M\times J^N_0M$, 
	$J^N_0g_0$ and $J^N_0g_1$ are given in local coordinates respectively by
	\begin{align*}
		(\alpha_0^j(0),\nabla \alpha_0^j(0),\dots, \nabla^{(N)}\alpha_0^j(0),0,\dots,0)_{j=1}^n&\in \left(\R\times \R^n\times\dotsm\times \R^{n^N}\right)^{2\times n}\quad\textrm{and}\\
		(0,\dots,0,\alpha_1^j(0),\nabla \alpha_1^j(0),\dots, \nabla^{(N)}\alpha_1^j(0))_{j=1}^n&\in \left(\R\times \R^n\times\dotsm\times \R^{n^N}\right)^{2\times n}.
	\end{align*}
	Moreover, since $g_0(q)\wedge g_1(q)\neq 0$, without loss of generality we can  assume that
	\be
		\alpha_0^1(0)=\alpha_1^2(0)=1,\quad \alpha_0^j(0)=0\textrm{ if } j\neq 1,\quad \alpha_1^j(0)=0\textrm{ if } j\neq 2.
	\ee
	
	Let  the codimension of $\mathcal{B}'_0$ in $T_0^*M\times J_0^NM\times J_0^NM$  be equal to $l$, and assume that $S_{l+1}$ is of the form $S_{(\pm I_l)}$. In particular, the degree of $S_{l+1}$ is maximal among $\{\deg(S_1),\dots,\deg(S_{l+1})\}$. Following Remark~\ref{r:unnumbered}, let us write the decomposition $$g_{I_{l+1}}=g_{I_{l+1,1}}+\dots+g_{I_{l+1,k}},$$ where we recall that $I_{l+1,1}$ is uniquely identified by the requirement that it contains the maximal number, say $s$, of occurrences of the letter $0$. Writing the analogous decomposition for simple relations $$S_{I_{l+1}}=S_{I_{l+1,1}}+\dots+S_{I_{l+1,k}},$$ we see that the coordinate expression of $S_{I_{l+1,1}}$ at $(\lambda,j^N_0(g_0),j^N_0(g_1))\in T^*_0M\setminus\{0\}\times J^N_0M\times J^N_0M$ takes the form 
	$$0=\langle\lambda,g_{I_{l+1,1}}(0)\rangle=\sum_{j=1}^n\lambda_j\partial_{x_1}^s\partial_{x_2}^{\deg(S_{l+1})-s-1}\alpha_1^j(0)+P_{I_{l+1,1}}(\lambda,j^N_0(g_0),j^N_0(g_1)),$$
	where $P_{I_{l+1,1}}$ is a polynomial expression in the coordinates of $\lambda$, $j_0^N(g_0)$ and $j_0^N(g_1)$ that does not contain any term of the form $\partial_{x_1}^s\partial_{x_2}^{\deg(S_{l+1})-s-1}\alpha_1^j(0)$, for $1\leq j\leq n$. By construction, these terms do not appear in any of the other summands $\langle\lambda,g_{I_{l+1,i}}\rangle$, for $i\neq 1$, neither among all other simple relations $S_{1},\dots, S_{l}$. Therefore, as $\lambda\neq 0$, we infer the existence of a further independent relation, and we conclude that $$\mathrm{codim}_{T_0^*M\times J_0^NM\times J_0^NM}(\mathcal{B}'_0\cap\{S_{l+1}=0\})=l+1.$$
	
	The case in which $S_{l+1}=S_{(-I_{l-1})}$ can be tackled similarly. In this situation $\deg(S_l)=\deg(S_{l+1})>\deg(S_i)$ for every $i<l$. 
	We may again exploit Remark~\ref{r:unnumbered}, and isolate the terms $g_{I_{l,1}},g_{I_{l,2}}$ and $g_{I_{{l+1},1}},g_{I_{{l+1},2}}$ in the decompositions of $g_{I_l}$ and $g_{I_{l+1}}$ respectively. 
	Observe that, by definition of $I_l$ and $I_{l+1}$, one has  $g_{I_{{l+1},1}}=g_{I_{{l},1}}$ and $g_{I_{{l+1},2}}=-g_{I_{{l},2}}$. Moreover, 
	$0$ appears $s$ times in $I_{l,1}$,  while $1$ appears $t$ times in $I_{l,2}$, and both $s$ and $t$ are maximal among their corresponding decompositions, so that we can write 
	\begin{align*}
		0=&\langle\lambda,g_{I_{l}}(0)\rangle=
\langle\lambda,g_{I_{l,1}}(0)\rangle+\langle\lambda,g_{I_{l,2}}(0)\rangle		+P_{I_{l}}(\lambda,j^N_0(g_0),j^N_0(g_1))
		\\=&\sum_{j=1}^n\lambda_j\Big(\partial_{x_1}^s\partial_{x_2}^{\deg(S_{l})-s-1}\alpha_1^j(0)+\partial_{x_1}^{\deg(S_{l})-t-1}\partial_{x_2}^t\alpha_0^j(0)\Big)+Q_{I_{l}}(\lambda,j^N_0(g_0),j^N_0(g_1)),\\
		0=&\langle\lambda,g_{I_{l+1}}(0)\rangle
		=
\langle\lambda,g_{I_{l,1}}(0)\rangle-\langle\lambda,g_{I_{l,2}}(0)\rangle+P_{I_{l+1}}(\lambda,j^N_0(g_0),j^N_0(g_1))\\
		=&\sum_{j=1}^n\lambda_j\Big(\partial_{x_1}^s\partial_{x_2}^{\deg(S_{l})-s-1}\alpha_1^j(0)-\partial_{x_1}^{\deg(S_{l})-t-1}\partial_{x_2}^t\alpha_0^j(0)\Big)+Q_{I_{l+1}}(\lambda,j^N_0(g_0),j^N_0(g_1)),
	\end{align*}
	where  $P_{I_{l}},P_{I_{l+1}},Q_{I_{l}},Q_{I_{l+1}}$ are polynomial expressions in the coordinates of $\lambda$, $j_0^N(g_0)$ and $j_0^N(g_1)$ that do not contain any term of the form $\partial_{x_1}^s\partial_{x_2}^{\deg(S_{l})-s-1}\alpha_1^j(0)$ and $\partial_{x_1}^{\deg(S_{l})-t-1}\partial_{x_2}^t\alpha_0^j(0)$, for $1\leq j\leq n$.
	In addition, these two terms are neither found among all other simple relations $S_{1},\dots, S_{l-1}$. Thus, 
		as $\lambda\neq 0$, 
		the relations $\langle \lambda, g_{I_l}(0)\rangle=0$ and $\langle \lambda, g_{I_{l+1}}(0)\rangle=0$ are mutually independent (since their gradients are not parallel)
and also independent from $\langle \lambda, g_{I_k}(0)\rangle=0$, $k=1,\dots, l-1$.
\end{proof}

\begin{lemma}\label{lemma:codim}
	Let $l,d_1,\dots, d_l\in\N\setminus\{0\}$ and consider $l$ words $I_1\in\mathfrak{A}^{d_1},\dots,I_l\in\mathfrak{A}^{d_l}$ with $d_j<d_{l-1}$ for every $j<l-1$ and $d_{l-1}=d_l$. Suppose that there exists $j<l-1$ such that $I_{l-1}=(0\,I_{j})$ and $I_l=(1\,I_j)$. 
	Using the notations introduced in \eqref{eq:simple_relation} and \eqref{eq:polynomial_relation}, consider the 
	family of polynomial relations $Q_r$, $r\in\N\setminus\{0\}$, constructed inductively using the simple relations $S_{I_1},\dots,S_{I_l}$ as follows
	\begin{align}\label{eq:Q}
Q_1=\det\left(\begin{array}{cc}
			\{S_0,S_{I_l}\} & \{S_1,S_{I_l}\}\\
			\{S_0,S_{I_{l-1}}\} & \{S_1,S_{I_{l-1}}\}
		\end{array}
		\right), 
		\quad 
	Q_r=\det\left(\begin{array}{cc}
		\{S_0,S_{I_l}\} & \{S_1,S_{I_l}\}\\
		\{S_0,Q_{r-1}\} & \{S_1,Q_{r-1}\}
		\end{array}
		\right)
		\ \textrm{ for }\: r\ge 2.
	\end{align}
	Fix $h\in\N$, an integer $N> d_l+h$, and define the set $\mathcal{B}\subset T^*M\times J^NM\times J^NM$ by 
	$$\mathcal{B}=\left\{(\lambda, j^N_q(f_0),j_q^N(f_1))\,\left|\,q=\pi(\lambda),\: \begin{aligned} &S_{I_1}(\lambda, j^N_q(f_0),j_q^N(f_1))=\dots=S_{I_l}(\lambda, j^N_q(f_0),j_q^N(f_1))=0\\
	&Q_1(\lambda, j^N_q(f_0),j_q^N(f_1))=\dots=Q_h(\lambda, j^N_q(f_0),j_q^N(f_1))=0
	\end{aligned}\right.\right\}.$$
	If 
	$(q(\cdot), u(\cdot),\lambda(\cdot))$ is an extremal triple on $[0,T]$,
	and if the sequence $\{t_i\}_{i\in\N}\subset[0,T]$ is  such that
	\begin{itemize}
		\item [i)] $(\lambda(t_i), j^N_{q(t_i)}(f_0),j_{q(t_i)}^N(f_1))\in\mathcal{B}$ for every $i\in\N$,
		\item [ii)] there exists $t_\infty=\lim_{i\to\infty}t_i$,
	\end{itemize}
	then, setting
	$$I_{l+1}=(0I_l),\quad I_{l+2}=(1I_l),
	$$
	either 
	$$(\lambda(t_\infty), j^N_{q(t_\infty)}(f_0),j_{q(t_\infty)}^N(f_1))\in\mathcal{B}\cap\{S_{I_{l+2}}\neq 0\}\cap\{Q_{h+1}=0\},$$
	or
	$$(\lambda(t_\infty), j^N_{q(t_\infty)}(f_0),j_{q(t_\infty)}^N(f_1))\in\mathcal{B}\cap\{S_{I_{l+1}}=0\}\cap\{S_{I_{l+2}}=0\}.$$
	Finally, defining for every $q\in M$ the set $\mathcal{B}'_q\subset T^*_qM\times J_q^NM\times J_q^NM$ by  
	\begin{align*}
\mathcal{B}'_q
	=\bigg\{(\lambda, j^N_q(g_0),j_q^N(g_1))\mid \lambda&\in T^*_q M\setminus\{0\},\;(g_0,g_1)\in {\rm Vec}(M)^2,\;g_0(q)\wedge g_1(q)\neq 0,\\
&	S_{I_1}(\lambda, j^N_q(g_0),j_q^N(g_1))=\dots=S_{I_l}(\lambda, j^N_q(g_0),j_q^N(g_1))=0\bigg\},
\end{align*}
	if the codimension of $\mathcal{B}'_q$ in $T_q^*M\times J_q^NM\times J_q^NM$ is equal to $l$, then
\begin{align*}
\mathrm{codim}_{T_q^*M\times J_q^NM\times J_q^NM}(\mathcal{B}'_q\cap\{S_{I_{l+2}}\neq 0\}\cap\{Q_{1}=0\}\cap\dots\cap\{Q_{h+1}=0\})&=l+h+1,\\
\mathrm{codim}_{T_q^*M\times J_q^NM\times J_q^NM}(\mathcal{B}'_q\cap\{S_{I_{l+1}}=0\}\cap\{S_{I_{l+2}}=0\})&=l+2.
\end{align*}
\end{lemma}

\begin{proof}
	 The proof of the first part of the statement follows along the same lines of Lemma \ref{lemma:Jets}, using equation \eqref{eq:integral_mean} 
	 both on $S_{I_l}$ and on $Q_h$, with the convention that $Q_0=S_{I_{l-1}}$. 
We prove in this way 
that the relations 
	\be\label{eq:proQ}
\{S_0,S_{I_{l}}\}+\bar{u}\{S_1,S_{I_{l}}\}=0	\quad \mbox{and} \quad \{S_0,Q_h\}+\bar{u}\{S_1,Q_h\}=0
	\ee
	 hold at $(\lambda(t_\infty), j^N_{q(t_\infty)}(f_0),j_{q(t_\infty)}^N(f_1))$, where the value $\bar{u}$ is the same in both identities, since it is computed as the limit of a common sequence. 
	If $S_{I_{l+2}}=\{S_1,S_{I_{l}}\}$ vanishes on the triple $(\lambda(t_\infty), j^N_{q(t_\infty)}(f_0),j_{q(t_\infty)}^N(f_1))$, then so does $S_{I_{l+1}}=\{S_0,S_{I_{l}}\}$. 
	From equation~\eqref{eq:proQ} we also deduce  that $(1,\bar{u})$ is in the kernel of $$\left(\begin{array}{cc}
	\{S_0,S_{I_l}\} & \{S_1,S_{I_l}\}\\
	\{S_0,Q_h\} & \{S_1,Q_h\}
	\end{array}
	\right),$$ and therefore that its determinant $Q_{h+1}$ vanishes at $(\lambda(t_\infty),j^N_{q(t_\infty)}(f_0),j^N_{q(t_\infty)}(f_1))$. 
	
	In order to prove the second part of the statement, as in Lemma~\ref{lemma:codim} the idea is to express all relations in local coordinates around $q$ on the product space $T^*_qM\times J_q^NM\times J_q^NM$, with the non-restrictive hypothesis that $g_0(0)=\partial_{x_1}$ and $g_1(0)=\partial_{x_2}$. Notice that 
	for what concerns the codimension of 
	$\mathcal{B}'_q\cap\{S_{I_{l+1}}=0\}\cap\{S_{I_{l+2}}=0\}$
	we can reason 
	exactly as in Lemma~\ref{lemma:codim}, since we deal in fact only with simple relations. 
	Thus we are left with the task of proving that, if $S_{I_{l+2}}\neq 0$, each polynomial relation $Q_{r}$ provides a condition independent  from $S_{I_1},\dots,S_{I_l}$ and $Q_1,\dots,Q_{r-1}$.

By construction, $Q_r$ is a polynomial relation in the variables $S_{(AB)}$, where $(AB)$ is the concatenation of a word $A$ of length at most $r$  with letters in $\{0,1\}$ and a word $B$ equal either to $I_{l-1}$ or $I_{l}$. 
	It is not hard to show, by induction, that $$Q_r=(-1)^r(S_{I_{l+2}})^r
	\mathrm{ad}^r_{S_0}(S_{I_{l-1}})+Q_r'=(-1)^rS_{(1I_{l})}^r
	S_{(0\cdots0I_{l-1})}+Q_r',$$ where $\mathrm{ad}^r_{S_0}$  denotes the iterated Poisson bracket with $S_0$
	 and $Q_r'$ is a polynomial relation in the same variables as $Q_r$ except for $\mathrm{ad}^r_{S_0}(S_{I_{l-1}})$.
	 Following Remark \ref{r:unnumbered}, we further decompose $f_{I_{l-1}}$ as $f_{I_{l-1},1}+\dots+f_{I_{l-1},k}$, where the letter $0$ appears in $I_{l-1,1}$ the maximal number of times, say $s$, among the collection $\{I_{l-1,i}\}_{i=1}^k$. In coordinates  we then write $$\mathrm{ad}^r_{S_0}(S_{I_{l-1}})(\lambda(t_\infty))=\sum_{j=1}^n\lambda_j\partial_{x_1}^{r+s}\partial_{x_2}^{\deg(S_{I_{l-1}})-s-1}\alpha_1^j(0)+P_{I_{l-1}}(\lambda,j_0^N(g_0),j_0^N(g_1)),$$ where $P_{I_{l-1}}$ is a polynomial expression in $\lambda$, $j_0^N(g_0)$ and $j_0^N(g_1)$ that does not contain any term of the form $\partial_{x_1}^{r+s}\partial_{x_2}^{\deg(S_{I_{l-1}})-s-1}\alpha_1^j(0)$. Since $\lambda\neq 0$ and the above is true for any $r\in\N$, we conclude that, as soon as $S_{I_{l+2}}\neq 0$, each $Q_r$ gives  a new independent condition, and 
	the claim on the codimension follows.
\end{proof}

\subsection{Collinear case}

The computation of the codimension of the  sets $\mc{B}'_q$ identified in Lemmas~\ref{lemma:Jets} and \ref{lemma:codim} relies on the linear independence at $q$ of $f_0(q)$ and  $f_1(q)$. 
We study in this section what happens when the condition $f_0(q)\wedge f_1(q)\ne 0$ fails to hold. 

We associate with the pair $(f_0,f_1)\in\mathrm{Vec}(M)^2$ the collinearity set 
\begin{equation}\label{defi:C}
\mathcal{C}=\{q\in M\mid f_0(q)\wedge f_1(q)=0\}.
\end{equation}

\begin{lemma}\label{lemma:dependent}
	Let $u\in L^\infty([0,T],[-1,1])$ and 
	$q:[0,T]\to M$ be a trajectory of the control system \eqref{eq:si} associated with 
	the control $u$. 
	Assume that $t_\infty\in[0,T]$ is such that $q(t_\infty)\in\{q\in M\mid 
	f_1(q)\wedge [f_0,f_1](q)\neq 0 \}$ and that 
there exists a sequence $\{t_i\}_{i\in\N}\subset [0,T]$ converging to $t_\infty$ 
	 such that
  $q(t_i)\in\mathcal{C}$ for every $i\in\N$.
  Then 
there exists 
$$\bar{u}:=\lim_{i\to\infty}\frac{1}{t_\infty-t_{i}}\int_{t_{i}}^{t_\infty}u(\tau)d\tau\in[-1,1]
$$
and   
  $f_0(q(t_\infty))+\bar u f_1(q(t_\infty))=0$. 
\end{lemma}
\begin{proof}
First notice that, by continuity, 
 $f_0(q(t_\infty))\wedge f_1(q(t_\infty))= 0$. 
Moreover,  since  $f_1(q(t_\infty))\wedge [f_0,f_1](q(t_\infty))\neq 0$, the set
$\mathcal{C}$ 
is, locally around $q(t_\infty)$, contained in an embedded $(n-1)$-dimensional manifold $\hat{\mathcal{C}}$ transversal to the 
vector field $f_1$. This can be seen, for instance, by choosing a local system of coordinates $(x_1,\dots,x_n)$ such that $f_1=\partial_{x_1}$ near $q(t_\infty)$. Write $f_0(x)=\sum_{i=1}^n a_i(x) \partial_{x_i}$. Then $\mathcal{C}$ is locally described by the conditions  $a_2(x)=\dots=a_n(x)=0$. Furthermore, up to restricting the coordinate chart, the condition 
$f_1(q(t_\infty))\wedge [f_0,f_1](q(t_\infty))\neq 0$ implies that there exists $j\in\{2,\dots,n\}$ such that $\partial_{x_1}a_j$ is nowhere vanishing. In particular, $\mathcal{C}$ is locally contained in the manifold $\hat{\mathcal{C}}=\{x\mid a_j(x)=0\}$, which is transversal to $f_1$. 

Let us take any coordinate system around $q(t_\infty)$. 
Notice that any converging subsequence of $\frac{q(t_\infty)-q(t_i)}{t_\infty-t_i}$ is tangent to 
$\hat{\mathcal{C}}$. 
Writing 
$$	\frac{q(t_\infty)-q(t_i)}{t_\infty-t_i}=\frac 1{t_\infty-t_i}\int_{t_i}^{t_\infty} (f_0(q(\tau))+u(\tau)f_1(q(\tau)))d\tau,$$
we deduce that for every converging subsequence of $\{\frac 1{t_\infty-t_i}\int_{t_i}^{t_\infty}u(\tau)d\tau\}_{i\in\N}$, its limit $\tilde u$ is such that $f_0(q(t_\infty))+\tilde u f_1(q(t_\infty))$ is tangent to $\hat{\mathcal{C}}$.
	The proof is concluded by noticing that, by transversality of $\hat{\mathcal{C}}$ and $f_1$,
 the only vector of the form $f_0(q(t_\infty))+\tilde u f_1(q(t_\infty))\in {\rm span}(f_1(q(t_\infty)))$ which 
is tangent to $\hat{\mathcal{C}}$ is zero. 
\end{proof}

\begin{remark}
The lemma says in particular that for every $(f_0,f_1)\in\mathrm{Vec}(M)^2$ and every trajectory $q:[0,T]\to M$ of \eqref{eq:si}, almost everywhere on $\{t\in [0,T]\mid q(t)\in\mathcal{C},\;f_1(q)\wedge [f_0,f_1](q)\ne 0\}$ we have $\dot q=0$.
This result is in the same spirit as \cite[Theorem~2.1]{CJT-SIAM}, where the  multi-input
case in considered.  
\end{remark}

\begin{defi}\label{d:Omega}
	For any extremal triple $(q(\cdot), u(\cdot),\lambda(\cdot))$ on $[0,T]$ of the time-optimal control problem \eqref{eq:si}, 
	we call
	$\Omega=\{t\in [0,T]\mid q(t)\in\mathcal{C},\;h_1(t)=0\}$.
	Moreover, we denote by $\Omega_0$ the set of all isolated points in $\Omega$, and inductively we declare $\Omega_k$ to be the set of isolated points in $\Omega\setminus(\bigcup_{j=0}^{k-1}\Omega_j)$.
\end{defi}

\begin{theorem}\label{thm:collinear}
	Let $(f_0,f_1)\in\mathrm{Vec}(M)^2$ and let $(q(\cdot), u(\cdot),\lambda(\cdot))$ be any extremal trajectory on $[0,T]$ of the time-optimal control problem \eqref{eq:si}. Assume that there exist a sequence $\{t_i\}_{i\in\N}\subset [0,T]$ and an integer $k\geq 0$ such that
	\begin{itemize}
		\item [a)] $t_i\in\Omega\setminus(\bigcup_{j=0}^k\Omega_j)$ for every $i\in\N$,
		\item [b)] there exists $t_\infty=\lim_{i\to\infty}t_i$ and $q(t_\infty)\in\{q\in M\mid f_1(q)\wedge [f_0,f_1](q)\neq 0\}$.
	\end{itemize}
	Then there exists $a\in [-1,1]$ such that, with the notation $s_a(\lambda):=\langle \lambda,(f_0+af_1)(\pi(\lambda))\rangle$, we have
	\be\label{eq:relations}
\mathrm{ad}^j_{s_a}(s_1)(\lambda(t_\infty))
=0,\quad \textrm{for every }\;0\leq j\leq k+2.
	\ee
\end{theorem}

\begin{proof}
	We proceed by induction on $k$, and we begin with the  case $k=0$. 
	First notice that 
	for 
	$t\in \Omega$ 
	both	$h_0(t)=0$ and $h_1(t)=0$. Hence, also $s_1(\lambda(t))=0$. 
	By continuity and by Rolle's theorem,  $\{s_{0},s_1\}(\lambda(t))=h_{01}(t)=0$ for every $t\in \Omega\setminus\Omega_0$. 
	Also notice that $\{s_{0},s_1\}={\rm ad}_{s_a}s_1$ for every $a\in [-1,1]$.
	Moreover, by item b) and Lemma \ref{lemma:dependent},
	there exists 
$$a=\lim_{i\to\infty}\frac{1}{t_\infty-t_{i}}\int_{t_{i}}^{t_\infty}u(\tau)d\tau\in [-1,1]
$$
and   
  $f_0(q(t_\infty))+a f_1(q(t_\infty))=0$.

From the identity
	\begin{align*}
		0&=\frac{1}{t_\infty-t_i}\int_{t_i}^{t_\infty}\frac{d}{d\tau}\{s_0,s_1\}(\lambda(\tau))d\tau\\ &=\frac{1}{t_\infty-t_i}\int_{t_i}^{t_\infty}\left(\{s_0,\{s_0,s_1\}\}(\lambda(\tau))+u(\tau)\{s_1,\{s_0,s_1\}\}(\lambda(\tau))\right)d\tau,
	\end{align*}
	which is valid for every $i\in\N$, passing to the limit as $i\to\infty$ we deduce the further relation $\mathrm{ad}_{s_{a}}^2 s_1(\lambda(t_\infty))=\mathrm{ad}_{s_{a}}\{s_0,s_1\}(\lambda(t_\infty))=0$.
	
	Assume now that the theorem holds for some $k\in\N$, and consider any sequence of points $\{t_i\}_{i\in\N}\in\Omega\setminus(\bigcup_{j=0}^{k+1}\Omega_j)$ satisfying items a) and b). 
	Apply Lemma~\ref{lemma:dependent}  and define  $a$ as above. 
The conclusion comes from noticing that 
	\begin{align}\label{eq:limdep2}
		0&=\frac{1}{t_\infty-t_i}\int_{t_i}^{t_\infty}\frac{d}{d\tau}\mathrm{ad}^{k+1}_{s_{a}}(s_{1})(\lambda(\tau))d\tau\nonumber\\ &=\frac{1}{t_\infty-t_i}\int_{t_i}^{t_\infty}
		(\mathrm{ad }_{s_0}\mathrm{ad}^{k+1}_{s_{a}}(s_{1})+u(\tau) \mathrm{ad }_{s_1}\mathrm{ad}^{k+1}_{s_{a}}(s_{1})
		)(\lambda(\tau))d\tau\\
		& \to  \mathrm{ad}^{k+2}_{s_{a}}(s_{1})(\lambda(t_\infty))\mbox{ as $i\to\infty$}.
	\end{align}
\end{proof}

Inspired by the arguments of \cite[Definition 4 and Lemma 4]{BonnardKupka}, we are now in the position of deducing quantitative estimates on the possible accumulations of points of $\Omega$ within the collinearity set $\mathcal{C}$.

\begin{lemma}\label{lemma:collinear}
	Let $q\in M$ and $N=n-1$. 
	Let us define the following two subsets of $J_q^NM\times J_q^NM$:
	\begin{align*}
		\mathcal{L}'&=\big\{ (j_q^N(f_0),j_q^N(f_1))\in J^N_qM\times J^N_qM\mid \mathrm{dim}\left(\mathrm{span}\{f_0(q),f_1(q),[f_0,f_1](q)\}\right)\leq 1\big\},\\
		\mathcal{L}''&=\big\{ (j_q^N(f_0),j_q^N(f_1))\in J^N_qM\times J^N_qM\mid f_1(q)\neq 0,\:\exists\, a\in\R\textrm{ such that }f_0(q)=af_1(q)\\&\hphantom{xxxxxxxx}\textrm{and }\mathrm{dim}\left(\mathrm{span}\{\mathrm{ad}^i_{f_0+a f_1}(f_1)(q)\mid 0\leq i\leq n-1\}\right)<n\big\}.
	\end{align*}
	Then $$\mathrm{codim}_{J^N_qM\times J^N_qM}\mathcal{L}'=2n-2\quad\textrm{and}\quad\mathrm{codim}_{J^N_qM\times J^N_qM}\mathcal{L}''=n.$$
\end{lemma}

\begin{proof}
	The first assertion is clear. For the second one just notice that for every $a\in \R$, the dimension of $\mathrm{span}\{\mathrm{ad}^i_{f_0+af_1}(f_1)(q)\mid 0\leq i\leq n-1\}$ is smaller than $n$ if and only if, in coordinates, 
	$$\det(H)=0,\qquad \mbox{ with }\quad H=\left(f_1,\dots,\mathrm{ad}^{n-1}_{f_0+af_1}(f_1)\right).$$ 
	The latter condition, taking as $a$ the unique scalar such that $f_0(q)+a f_1(q)=0$, identifies a set of codimension one inside 
	$$\mathcal{D}:=\left\{(j_q^N(f_0),j_q^N(f_1))\in J^N_qM\times J^N_qM\mid f_1(q)\neq 0, f_0(q)\wedge f_{1}(q)=0\right\}.$$ 
	Summing it up, we deduce that 
	\begin{align*}
		\mathrm{codim}_{J^N_qM\times J^N_qM}\mathcal{L}''&=\mathrm{codim}_{J^N_qM\times J^N_qM}\mathcal{D}\\&+\mathrm{codim}_{\mathcal{D}}\left\{(j_q^N(f_0),j_q^N(f_1))\in J^N_qM\times J^N_qM\mid\det(H)=0 \right\}\\&=(n-1)+1=n.
	\end{align*}
\end{proof}

\begin{corollary}\label{cor:collinear}
Let $n\ge 2$.
	For a generic pair $(f_0,f_1)\in\mathrm{Vec}(M)^2$ and for every extremal trajectory of the time-optimal control problem \eqref{eq:si}, 
	we have 
	$\Omega=\Omega_0\cup\dots\cup \Omega_{n-2}$, where $\Omega$ and $\Omega_j$ are defined as in Definition~\ref{d:Omega}.
\end{corollary}

\begin{proof}
	If along  an extremal triple $(q(\cdot),u(\cdot),\lambda(\cdot))$ there exists $t\in\Omega\setminus(\bigcup_{j=0}^{n-3}\Omega_j)$, which is not isolated in this set and such that $q(t)\in\{q\in M\mid f_1(q)\wedge f_{01}(q)\neq 0\}$, then by Theorem~\ref{thm:collinear} $\lambda(t)$ annihilates $\mathrm{ad}^i_{f_0+af_1}(f_1)(q(t))$ for every $0\leq i\leq n-1$, where $a$ is the proportionality coefficient between $-f_0(q(t))$ and $f_1(q(t))$. By Lemma~\ref{lemma:collinear} and Thom's 	transversality theorem (see, e.g., 
	\cite{abraham-robbin,goresky2012stratified}),  
		for a generic pair $(f_0,f_1)\in\mathrm{Vec}(M)^2$ this is possible only at isolated points of $M$. 
		Equivalently, for a generic pair $(f_0,f_1)\in\mathrm{Vec}(M)^2$ the set $\Omega$ is equal to $\bigcup_{j=0}^{n-2}\Omega_j$. 
	On the other hand, another application of Thom's transversality theorem  
	says that,  for a generic choice of $(f_0,f_1)$, the points $q\in M$ such that $f_0(q)\wedge f_{1}(q)=0$ and $f_1(q)\wedge f_{01}(q)=0$ are isolated (since $2n-2\geq n$ when $n\ge2$). 
	 This concludes the proof. 
\end{proof}

\section{Proof of Theorem \ref{t:main}}\label{sec:proof}

Theorem~\ref{t:main} directly follows from Theorem~\ref{thm:start}
and Proposition~\ref{p:intermezzo} below. 

\begin{proposition}\label{p:intermezzo}
There exists 
an open and dense set $\mc{V}\subset \mathrm{Vec}(M)^2$ such that, 
		for any pair 
		$(f_0,f_1)\in\mc{V}$ and 
for any extremal triple $(q(\cdot),u(\cdot),\lambda(\cdot))$ on $[0,T]$ 
of the time-optimal control problem \eqref{eq:si},
 the set 
 $$\Xi=\{t\in [0,T]\mid h_1(t)=h_{01}(t)=h_{+01}(t)= 0\mbox{ or }h_1(t)=h_{01}(t)=h_{-01}(t)= 0\}$$ 
 satisfies $\Xi=\Xi_1\cup \dots \cup \Xi_{(n-1)^2}$, 
where $\Xi_{1}$ denotes the set of isolated points of $\Xi$ and
$\Xi_{j+1}$ denotes the set of isolated points of $\Xi\setminus \cup_{i=1}^j\Xi_i$ for $j\ge 1$.  
\end{proposition}
\begin{proof}
Let $k\in\N$, $(f_0,f_1)\in\mathrm{Vec}(M)^2$ and $(q(\cdot),u(\cdot),\lambda(\cdot))$ be a time-extremal trajectory of the time-optimal control problem \eqref{eq:si}. 
Let $t\in\Xi\setminus(\bigcup_{j=1}^k\Xi_j)$ and assume for now
that $f_0(q(t))\wedge f_1(q(t))\neq 0$. Owing to the fact that $t$ is an accumulation point for $\Xi\setminus(\bigcup_{j=1}^{k-1}\Xi_j)$ and reasoning iteratively, 
we identify a set $\{ t_{n_1,\dots,n_r}\mid r=1,\dots,k,\;n_1,\dots,n_r\in\N\}$ such that 
\begin{align*}
\lim_{n_1\to \infty} t_{n_1}&=t,\\
\lim_{n_r\to \infty} t_{n_1,\dots,n_r}&=t_{n_1,\dots,n_{r-1}},\quad & \mbox{for  $r=2,\dots,k$ and  $n_1,\dots,n_{r-1}\in\N$,}\\
 t_{n_1,\dots,n_r}&\in \Xi\setminus \bigcup_{j=1}^{k-r}\Xi_{j},&\mbox{for  $r=2,\dots,k$ and $n_1,\dots,n_r\in\N$,}
\\
 t_{n_1,\dots,n_k}&\in \Xi,&\mbox{for  $n_1,\dots,n_k\in\N$.}
 \end{align*}

Using repeatedly Lemmas \ref{lemma:Jets} and \ref{lemma:codim} and exploiting the fact that each application of one of the two lemmas yields a finite number of alternatives,  we deduce from a diagonal extraction argument that, up to taking suitable subsequences, 
\begin{itemize}
	\item[i)] There exist 
	$k+1$ sets $\mathcal{B}_0,\dots,\mathcal{B}_k
	\subset T^*M\times J^{k+2}M\times J^{k+2}M$ such that $$(\lambda(t), j^{k+2}_{q(t)}(f_0),j^{k+2}_{q(t)}(f_1))\in 
	\mathcal{B}_k$$ and 
	$$(\lambda(t_{n_1,\dots,n_r}), j^{k+2}_{q(t_{n_1,\dots,n_r})}(f_0),j^{k+2}_{q(t_{n_1,\dots,n_r})}(f_1))\in 
	\mathcal{B}_{k-r}$$ for every $r=1,\dots,k$, $n_1,\dots,n_r\in\N$. 
	\item [ii)] For every $0\leq r\leq k$, 
	$\mathcal{B}_r$ is defined by the vanishing of, say, $l_r$ simple relations and $m_r$ polynomial relations (using the terminology of Definition~\ref{def:relations}). Moreover, denoting $\mathcal{B}_r(q)=\mathcal{B}_r\cap T^*_{q}M\times J^{k+2}_{q}M\times J^{k+2}_{q}M$, we have
	\begin{align}
	\mathrm{codim}_{T^*_{q(t)}M\times J^{k+2}_{q(t)}M\times J^{k+2}_{q(t)}M}\mathcal{B}_k(q(t))&=l_k+m_k,\\ 
	\mathrm{codim}_{T^*_{q(t_{n_1,\dots,n_r})}M\times J^{k+2}_{q(t_{n_1,\dots,n_r})}M\times J^{k+2}_{q(t_{n_1,\dots,n_r})}M}\mathcal{B}_{k-r}(q(t_{n_1,\dots,n_r}))&=l_{k-r}+m_{k-r},
	\end{align}
	 for every $r=1,\dots,k$, $n_1,\dots,n_r\in\N$. 
	\end{itemize}
By construction, the set $\mathcal{B}_k(q)$ is homogeneous with respect to the first component.
To prove the proposition in the set $\{ q\in M \mid f_0(q)\wedge f_1(q)\neq 0\}$ it is then sufficient to show that there exists $K\le (n-1)^2$ such that if $k\ge K$, then  
there exists $r\in \{0,\dots,k\}$ such that the codimension $l_{k-r}+m_{k-r}$ 
of  $\mathcal{B}_{k-r}(q(t_{n_1,\dots,n_r}))$ in $T^*_{q(t_{n_1,\dots,n_r})}M\times J^{k+2}_{q(t_{n_1,\dots,n_r})}M\times J^{k+2}_{q(t_{n_1,\dots,n_r})}M$ 
is strictly larger than $2n-1$. Indeed, if this were true, then denoting by $\pi:T^*M\times J^{k+2}M\times J^{k+2}M\to J^{k+2}M\times J^{k+2}M$ the canonical projection, we would conclude by standard transversality arguments 
\cite{goresky2012stratified}
combined with
the inequality
\begin{align}
&\mathrm{codim}_{J^{k+2}_{q(t_{n_1,\dots,n_r})}M\times J^{k+2}_{q(t_{n_1,\dots,n_r})}M}\pi\left(\mathcal{B}_{k-r}(q(t_{n_1,\dots,n_r}))\right)\\
&\hphantom{xxxxx}\geq \mathrm{codim}_{T^*_{q(t_{n_1,\dots,n_r})}M\times J^{k+2}_{q(t_{n_1,\dots,n_r})}M\times J^{k+2}_{q(t_{n_1,\dots,n_r})}M}\mathcal{B}_{k-r}(q(t_{n_1,\dots,n_r}))-n+1>n,
\end{align}
where the term $+1$ 
is due to the homogeneity of $\mathcal{B}_{k-r}(q)$ with respect to the first component.

We introduce now a discrete dynamics on $\N^2$, which describes the admissible patterns 
of $r\mapsto (l_{r},m_{r})$. 
Define three mappings $F_0,F_1,F_2:\N^2\to\N^2$ by
$$F_0(x_1,x_2)=(x_1,x_2)+(1,0),\quad 
F_1(x_1,x_2)=(x_1,x_2)+(0,1), \quad 
F_2(x_1,x_2)=(x_1,0)+(2,0).
$$ 
We say that an admissible curve $\gamma$ of length $p\in \N$ for this dynamical system is a 
map $\gamma:\{0,\dots,p\}\to \N^2$  such that
	\begin{itemize} 
		\item[i)] $\gamma(0)=(3,0)$,
		\item[ii)] there exists $j\in\{1,\dots,p\}$ such that  $\gamma(i)=F_0(\gamma(i-1))$ for $i=1,\dots,j$ 
	and $\gamma(i)=F_{\sigma_i}(\gamma(i-1))$, with $\sigma_{i}\in\{1,2\}$, for $i=j+1,\dots,p$.
	\end{itemize}
Observe that the initial condition fixed in i) reflects the definition of $\Xi$, 
$F_0$ describes the creation of a new simple relation (Lemma~\ref{lemma:Jets}), while $F_1$ and $F_2$ encode the occurrence of, respectively, a new polynomial relation and two new simple relations (Lemma~\ref{lemma:codim}).

We are going to compute the minimal $K$ so that, for $k\ge K$, any admissible curve $\gamma$ of length $k$ exits the region $T:=\{(x_1,x_2)\in\N^2\mid x_1+x_2\leq 2n-1\}$.
It is not difficult to see that the longest admissible curve $\gamma$ staying in $T$ is as indicated in Figure~\ref{f:triangle}, that is, we apply once $F_0$, then $2n-5$ times $F_1$, then once $F_2$, then $2n-7$ times $F_1$, once $F_2$, and so on. 
The length of such curve $\gamma$ is equal to 
$$\textrm{length}(\gamma)=1+(2n-5)+1+(2n-7)+1+\dots+(2n-(2n-1))=(n-2)(n-1),$$ 
which implies that $K= 1+(n-2)(n-1)$.
\begin{figure}[h!]
	\includegraphics[scale=.8]{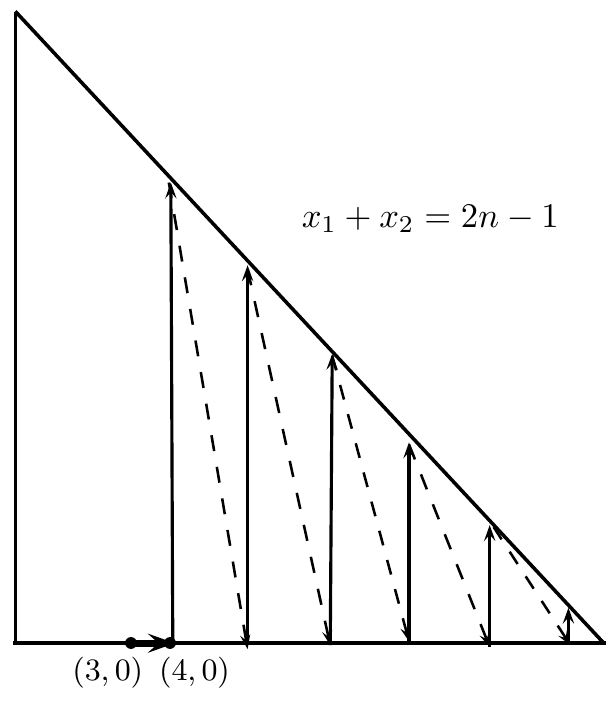}
	\caption{The longest admissible curve $\gamma$}\label{f:triangle}
\end{figure}

It just remains to explain what can happen inside the collinearity set $\mathcal{C}$ introduced in \eqref{defi:C}:
for a generic choice of $(f_0,f_1)$, along any extremal trajectory the points of $\Omega$ can accumulate at most $n-2$ times according to 
Corollary~\ref{cor:collinear}. On the other hand any point of $\Omega$ is itself an element of $\Xi\setminus(\bigcup_{j=1}^{K}\Xi_j)$ at worst, which implies that the order of the Fuller points can increase at most by $n-2$ within $\mathcal{C}$. This concludes the proof of Proposition~\ref{p:intermezzo} since $K+n-2=1+(n-2)(n-1)+n-2= (n-1)^2$. 
\end{proof}

\section{Time-optimal trajectories in dimension 
$n=3$}\label{sec:3}

We devote this section to a more careful analysis of Fuller times for \emph{time-optimal} 
(and not only extremal)
trajectories,
in the case of a three dimensional manifold $M=M^3$. In fact, for a time-optimal trajectory there are powerful second-order techniques \cite{AgrachevGamkrelidze} that permit us to be a bit sharper in our estimate on the maximal order of Fuller points, at least if we just focus on this smaller class of curves.
 By Theorem~\ref{t:main}, we already know the upper bound $(3-1)^2=4$. 
The main result of this section is the following.

\begin{theorem}\label{thm:sharp}
	For a generic pair $(f_0,f_1)\in\mathrm{Vec}(M)^2$, none of the time-optimal trajectories of the 
	control system \eqref{eq:si} 
	has Fuller times of order greater than two.
\end{theorem}

For the rest of this section we adopt the following convention:  for  any subset $\Theta\subset [0,T]$, we denote by $q(\Theta)$ its image along the trajectory $q(\cdot)$.

Let us fix then a time-optimal trajectory. We collect previous results from \cite{AgrachevSigalotti,KS89,SigalottiJMS} in the following statement.

\begin{prop}\label{prop:noFul}
	Let $(f_0,f_1)\in\mathrm{Vec}(M)^2$ and $q(\cdot)$ be any time-optimal trajectory of the  control system \eqref{eq:si}. Let us consider, with the notations of Definition~\ref{def:word}, the subsets
	\begin{align*}
		A_1&=\{q\in M\mid f_1(q)\wedge f_{01}(q)\wedge f_{+01}(q)\neq 0,\; f_1(q)\wedge f_{01}(q)\wedge f_{-01}(q)\neq 0\},\\
		A_2&=\{q\in M\mid f_1(q)\wedge f_{01}(q)\wedge f_{+01}(q)= 0,\; f_1(q)\wedge f_{01}(q)\wedge f_{++01}(q)\neq 0,\\&\hphantom{=\{q\in M\mid \,\,\,}f_1(q)\wedge f_{01}(q)\wedge f_{-01}(q)\neq 0\},\\
		A_3&=\{q\in M\mid f_1(q)\wedge f_{01}(q)\wedge f_{-01}(q)= 0,\; f_1(q)\wedge f_{01}(q)\wedge f_{--01}(q)\neq 0,\\&\hphantom{=\{q\in M\mid \,\,\,}f_1(q)\wedge f_{01}(q)\wedge f_{+01}(q)\neq 0\},\\
		A_4&=\{q\in M\mid f_1(q)\wedge f_{01}(q)\wedge f_{+01}(q)= 0,\; f_1(q)\wedge f_{01}(q)\wedge f_{++01}(q)= 0,\\&\hphantom{=\{q\in M\mid \,\,\,}f_1(q)\wedge f_{01}(q)\wedge f_{+++01}(q)\neq 0,\; f_1(q)\wedge f_{01}(q)\wedge f_{-01}(q)\neq 0\},\\
		A_5&=\{q\in M\mid f_1(q)\wedge f_{01}(q)\wedge f_{-01}(q)= 0,\; f_1(q)\wedge f_{01}(q)\wedge f_{--01}(q)= 0,\\&\hphantom{=\{q\in M\mid \,\,\,}f_1(q)\wedge f_{01}(q)\wedge f_{---01}(q)\neq 0,\; f_1(q)\wedge f_{01}(q)\wedge f_{+01}(q)\neq 0\},\\
		A_6&=\{q\in M\mid f_1(q)\wedge f_{01}(q)= 0,\; f_1(q)\wedge f_{+01}(q)\wedge f_{-01}(q)\neq 0,\\&\hphantom{=\{q\in M\mid \,\,\,}f_1(q)\wedge f_{+01}(q)\wedge f_{++01}(q)\neq 0,\; f_1(q)\wedge f_{-01}(q)\wedge f_{--01}(q)\neq 0\}.
	\end{align*}
	If $q(t)\in\bigcup_{i=1}^6A_i$, then $t\not\in \Sigma\setminus\Sigma_0$. 
\end{prop}

Define now the set
\be\label{eq:W}
	W=\{q\in M\mid  f_1(q)\wedge f_{01}(q)\wedge f_{+01}(q)=0,\;f_1(q)\wedge f_{01}(q)\wedge f_{-01}(q)=0,\;f_1(q)\wedge f_{01}(q)\neq 0\}.
\ee
As a consequence of Proposition \ref{prop:noFul}, we can  infer the following result.

\begin{lemma}\label{lemma:outW}
	For a generic pair $(f_0,f_1)\in\mathrm{Vec}(M)^2$ and for every time-optimal trajectory $q(\cdot)$ of the control system \eqref{eq:si},
	 $q(\Sigma\setminus \Sigma_0)\setminus W$ is made of isolated points only.
\end{lemma}
\begin{proof}
	The result is proved by 
	using the same computational approach based on transversality theory as in the proof of Lemma~\ref{lemma:Jets}.
Instead of working in $T^*M$ as in Lemma~\ref{lemma:Jets},	it is actually sufficient to prove that $$\mathrm{codim}_{J^N_qM\times J^N_qM}\left(\bigcap_{i=1}^6 \mathcal{A}_i^c\cap \mathcal{W}^c\right)\geq 3,\qquad q\in M,$$ 
where $\mathcal{A}_1,\dots,\mathcal{A}_6$ and $\mathcal{W}$ are the subsets of $J^N M\times J^N M$ defined implicitly by the relations
$$A_i=\{ q\in M \mid (j^N_q(f_0),j^N_q(f_1))\in\mathcal{A}_i\},\quad i=1,\dots,6,\quad W=\{ q\in M \mid (j^N_q(f_0),j^N_q(f_1))\in\mathcal{W}\}.$$
 Pick then any point $q\in W^c$ that satisfies $f_1(q)\wedge f_{01}(q)=0$. Then $\mathcal{W}\cap J^N_qM\times J^N_qM$  
  is already a set of codimension two in $J^N_qM\times J^N_qM$. Moreover, if $q\in A_6^c$, then necessarily 
  the jets of $f_0,f_1$ at $q$ satisfy  another nontrivial dependence relation, and we can conclude.
	
	On the other hand, suppose that $q\in \cap_{i=1}^6 A_i^c$ and that $f_1(q)\wedge f_{01}(q)\wedge f_{+01}(q)\neq 0$, the remaining case being identical. Then since $q\in A_1^c$ we infer the relation $f_1(q)\wedge f_{01}(q)\wedge f_{-01}(q)=0$. We pass now to the condition $q\in A_3^c$, and we see that this obliges $f_1(q)\wedge f_{01}(q)\wedge f_{--01}(1)=0$. Finally, the relation $q\in A_5^c$ forces $f_1(q)\wedge f_{01}(q)\wedge f_{---01}(q)=0$, which in turn provides us with a third dependence relation at $q$, and therefore once again we conclude. 
\end{proof}

\begin{proof}[Proof of Theorem \ref{thm:sharp}] 
Lemma~\ref{lemma:outW} states, in particular,  that for a generic choice of the pair $(f_0, f_1)$ and for every time-optimal trajectory $q(\cdot)$ we have that $q(\Sigma\setminus\Sigma_0)\setminus W\subset q(\Sigma_1)
$, 
or equivalently that 
\be\label{eq:incl}
q(\Sigma\setminus (\Sigma_0\cup \Sigma_1))\subset W.\ee 

We are left to prove that 
the density points of $q(\Sigma\setminus (\Sigma_0\cup \Sigma_1))=q(\Sigma\setminus (\Sigma_0\cup \Sigma_1))\cap W$ are isolated.

We have already shown that along any time-extremal $(q(\cdot),u(\cdot),\lambda(\cdot))$, whenever $t\in\Sigma\setminus\Sigma_0$ the relations $$h_1(\lambda(t))=\langle \lambda(t), f_1(q(t))\rangle=0 \quad\textrm{and}\quad h_{01}(\lambda(t))=\langle \lambda(t), f_{01}(q(t))\rangle=0$$ hold true. Since, by definition, for every point $q\in W$ both $f_{+01}(q)$ and $f_{-01}(q)$ belong to the two-dimensional space $\mathrm{span}\{f_1(q),f_{01}(q)\}$, then for every $t\in\Sigma\setminus(\Sigma_0\cup\Sigma_1)$ also $h_{+01}(\lambda(t))=h_{-01}(\lambda(t))=0$. If $t_\infty$ is an accumulation point of $\Sigma\setminus(\Sigma_0\cup\Sigma_1)$, then, by Lemma~\ref{lemma:codim} and using the Jacobi identity, either $h_{0101}(t_\infty)=0$  or $h_{0101}(t_\infty)\ne 0$ and
\begin{equation}\label{eq:destt}
h_{0001}(t_\infty)h_{1101}(t_\infty)-h_{0101}(t_\infty)^2= 0.\end{equation}
When $h_{0101}(t_\infty)=0$, we conclude by transversality, noticing that 
$$f_{0101}(q(t_\infty))\in \lambda(t_\infty)^\perp=\mathrm{span}\{f_1(q(t_\infty)),f_{01}(q(t_\infty))\}$$
 provides a third independent condition on the jet of the pair $(f_0,f_1)$ at $q(t_\infty)$. 
In the case $h_{0101}(t_\infty)\ne 0$, 
let us define in a neighborhood of $q(t_\infty)$ a system of coordinates $(x_1,x_2,x_3)$ 
so that $(dx_1,dx_2,dx_3)$ is dual to $(f_1,f_{01},f_{0101})$. 
Then 
 \eqref{eq:destt}  says that the product of the 
 third components of  $ f_{0001}(q(t_\infty))$ and $f_{1101}(q(t_\infty))$ is equal to one, which corresponds to a third independent condition on the jet of the pair $(f_0,f_1)$  at $q(t_\infty)$.  
\end{proof}

\bibliographystyle{abbrv}
\bibliography{biblio}
	 
\end{document}